\documentclass[a4paper,reqno,10pt]{amsart}

\usepackage{amssymb}
\usepackage{amstext}
\usepackage{amsmath}
\usepackage{amscd}
\usepackage{amsthm}
\usepackage{amsfonts}
\usepackage{enumerate}
\usepackage{graphicx}
\usepackage{latexsym}
\usepackage{mathrsfs}
\usepackage[all]{xy}
\usepackage{mathtools}
\xyoption{all}

\usepackage{pstricks}
\usepackage{lscape}
\usepackage{comment}

\newtheorem{theorem}{Theorem}[section]
\newtheorem{proposition}[theorem]{Proposition}

\newtheorem{lemma}[theorem]{Lemma}
\newtheorem{question}[theorem]{Question}

\theoremstyle{definition}

\newtheorem{definition}[theorem]{Definition}
\newtheorem{definition-proposition}[theorem]{Definition-Proposition}

\newtheorem{remark}[theorem]{Remark}

\newcommand{\Hom}{\operatorname{Hom}\nolimits}
\renewcommand{\mod}{\mathsf{mod}\hspace{.01in}}

\newcommand{\dimvec}{\underline{\operatorname{dim}}\hspace{.01in}}

\newcommand{\End}{\operatorname{End}\nolimits}
\newcommand{\Ext}{\operatorname{Ext}\nolimits}
\newcommand{\op}{\operatorname{op}\nolimits}

\newcommand{\rad}{\operatorname{rad}\nolimits}
\newcommand{\soc}{\operatorname{soc}\nolimits}

\newcommand{\Ker}{\operatorname{Ker}\nolimits}

\renewcommand{\AA}{{\mathcal A}}
\newcommand{\TT}{{\mathcal T}}
\renewcommand{\SS}{{\mathcal S}}
\newcommand{\FF}{{\mathcal F}}

\newcommand{\CC}{{\mathbf C}}
\newcommand{\PP}{{\mathbf P}}
\newcommand{\QQ}{{\mathbf Q}}
\newcommand{\RR}{{\mathbf R}}

\newcommand{\add}{\mathsf{add}\hspace{.01in}}
\newcommand{\Fac}{\mathsf{Fac}\hspace{.01in}}
\newcommand{\Sub}{\mathsf{Sub}\hspace{.01in}}
\newcommand{\Filt}{\mathsf{Filt}\hspace{.01in}}

\newcommand{\tors}{\mbox{\rm tors}\hspace{.01in}}
\newcommand{\ftors}{\mbox{\rm f-tors}\hspace{.01in}}

\newcommand{\torf}{\mbox{\rm torf}\hspace{.01in}}
\newcommand{\ftorf}{\mbox{\rm f-torf}\hspace{.01in}}

\newcommand{\sttilt}{\mbox{\rm s$\tau$-tilt}\hspace{.01in}}

\newcommand{\commentout}[1]{}
\begin{document}
\title{Lattice structure of torsion classes for path algebras}
\author{Osamu Iyama, Idun Reiten, Hugh Thomas, Gordana Todorov}
\address{Graduate School of Mathematics, Nagoya University, Chikusa-ku, Nagoya. 464-8602, Japan}
\email{iyama@math.nagoya-u.ac.jp}
\urladdr{http://www.math.nagoya-u.ac.jp/~iyama/}
\address{Department of Mathematical Sciences, NTNU, Norway}
\email{idunr@math.ntnu.no}
\address{Department of Mathematics and Statistics, University of New Brunswick, Fredericton, NB, Canada, E3B 5A3}
\email{hugh@math.unb.ca}
\address{Department of Mathematics, Northeastern University, 360 Huntington Avenue, Boston, MA 02115}
\email{g.todorov@neu.edu}
\urladdr{http://www.math.neu.edu/}
\thanks{All of the authors were partially supported by MSRI.
The first author was supported by JSPS Grant-in-Aid for Scientific Research 24340004, 23540045 and 22224001.
The second author was supported by FRINAT grants 19660 and 231000 from the Research Council of Norway.
The third author was partially supported by an NSERC Discovery Grant.
The fourth author was partially supported by NSF grant DMS-1103813.}
\thanks{2010 {\em Mathematics Subject Classification.} 16G20, 18E40, 05E10.}
\thanks{{\em Key words and phrases.} Torsion class, lattice, $\tau$-tilting module, tilting module}

\begin{abstract} We consider module categories of path algebras of connected acyclic quivers. It is shown in this paper that the set of functorially finite torsion classes form a lattice if and only if the quiver is either Dynkin quiver of type A, D, E, or the quiver has exactly two vertices. 
\end{abstract}
\maketitle
\setcounter{section}{-1}

\section {Introduction}


Let $\Lambda$ be a finite dimensional algebra over an algebraically closed field $k$, and $\mod \Lambda$ the category of finite dimensional $\Lambda$-modules.
In this setup a subcategory $\TT$ is a \emph{torsion class} if it is closed under factor modules, isomorphisms and extensions.
The set $\tors \Lambda$ of torsion classes is a partially ordered set by inclusion,
and it is easy to see that it is always a lattice (see Definition \ref{definition of lattice}).
There is however an important subset $\ftors \Lambda$ of $\tors\Lambda$,
where $\ftors \Lambda$ denotes the set of torsion classes which are functorially finite in $\mod \Lambda$.
In this setting a torsion class is \emph{functorially finite} precisely when it is of the form $\Fac X$ for some $X$ in $\mod\Lambda$ \cite{AS}.
The set  $\ftors \Lambda$ is of special interest since the elements are in bijection with the support $\tau$-tilting modules (see Definition \ref{define tau-rigid}),
which were introduced in \cite{AIR}.
This bijection also induces a structure of partially ordered set on the support $\tau$-tilting modules.
A related partial order has been studied in classical tilting theory by many authors (e.g. \cite{RS,HU,AI,K}).
There is also a connection with the weak order on finite Coxeter groups \cite{M}. 

The aim of this paper is to study the following questions.

\begin{question}\label{motivating question}
Let $\Lambda$ be a finite dimensional $k$-algebra.
\begin{itemize}
\item[(a)] When is $\ftors\Lambda$ a complete lattice?
\item[(b)] When is $\ftors\Lambda$ a lattice?
\end{itemize}
\end{question}

A simple answer to Question \ref{motivating question}(a) is given in terms of the $\tau$-rigid finiteness (see Definition \ref{tau-rigid finite} for details):

\begin{theorem}\label{complete lattice}
Let $\Lambda$ be a finite dimensional $k$-algebra. Then the following conditions are equivalent.
\begin{itemize}
\item[(a)] $\ftors\Lambda$ forms a complete lattice.
\item[(b)] $\ftors\Lambda$ forms a complete join-semilattice.
\item[(c)] $\ftors\Lambda$ forms a complete meet-semilattice.
\item[(d)] $\ftors\Lambda=\tors\Lambda$ holds (i.e. any torsion class in $\mod\Lambda$ is functorially finite).
\item[(e)] $\Lambda$ is $\tau$-rigid finite.
\end{itemize}
\end{theorem}

On the other hand, Question \ref{motivating question}(b) for an arbitrary algebra $\Lambda$ above does not seem to have a simple answer. Hence
we are mainly concerned with $\ftors (kQ)$ where $kQ$ is the path algebra of a finite connected acyclic quiver $Q$. Our main theorem is the following.

\begin{theorem}\label{main}
Let $Q$ be a finite connected quiver with no oriented cycles. Then the following conditions are equivalent.
\begin{itemize}
\item[(a)] $\ftors (kQ)$ forms a lattice.
\item[(b)] $\ftors (kQ)$ forms a join-semilattice (see Definition \ref{definition of lattice}).
\item[(c)] $\ftors (kQ)$ forms a meet-semilattice.
\item[(d)] $Q$ is either a Dynkin quiver or has at most 2 vertices.
\end{itemize}
\end{theorem}

We remark that condition (d) is equivalent to the property that all the rigid indecomposable $kQ$-modules are preprojective or preinjective.

We also show the following result.

\begin{theorem}\label{canonical algebra}
Let $\Lambda$ be a concealed canonical algebra (in particular a canonical algebra) or a tubular algebra. Then the following conditions are equivalent.
\begin{itemize}
\item[(a)] $\ftors\Lambda$ forms a lattice.
\item[(b)] $\ftors\Lambda$ forms a join-semilattice.
\item[(c)] $\ftors\Lambda$ forms a meet-semilattice.
\item[(d)] $\Lambda$ has at most 2 simple modules up to isomorphism.
\end{itemize}
\end{theorem}



\medskip
The paper is organized as follows. In section 1 we give a proof of
Theorem \ref{complete lattice} and give an important criterion for deciding
if $\ftors\Lambda$ is a lattice, together with some preliminary results.
In subsection 2.1 we show our sufficient conditions for $\ftors (kQ)$ to be a lattice.
In subsection 2.2 we show that $\ftors(kQ)$ is not a lattice for a path algebra $kQ$ of an extended Dynkin quiver $Q$ with at least 3 vertices.
In subsection 2.3 we deal with a path algebra $kQ$ of a wild quiver $Q$ with 3 vertices, and show that $\ftors(kQ)$ is not a lattice.
In subsection 2.4 we put things together to prove Theorem \ref{main}.
In subsection 2.5 we prove Theorem \ref{canonical algebra}.

\medskip\noindent
{\bf Acknowledgements.}
The authors would like to thank Otto Kerner and Claus Michael Ringel for valuable discussions.
The authors are grateful to MSRI and Oberwolfach for having had the opportunity to work together in such inspiring environments.

\section{Lattice structure of torsion classes for finite dimensional algebras}

\subsection{General results}

Let $\Lambda$ be a finite dimensional $k$-algebra.
A full subcategory $\FF$ of $\mod\Lambda$ is a \emph{torsionfree class} if it is closed under submodules, isomorphisms and extensions.
We denote by $\torf \Lambda$ the set of all torsionfree classes in $\mod \Lambda$,
and by $\ftorf \Lambda$ the set of all \emph{functorially finite} torsionfree classes
(i.e. torsionfree classes of the form $\Sub X$ for some $X\in\mod\Lambda$).
The following observation is classical.

\begin{proposition}\label{tors and torf}
\begin{itemize}
\item[(a)] We have a bijection
\[\tors \Lambda\to\torf\Lambda\ \ \ \TT\mapsto\TT^\perp:=\{X\in\mod \Lambda\mid\Hom_\Lambda(\TT,X)=0\}\]
whose inverse is given by
\[\torf \Lambda\to\tors \Lambda\ \ \ \FF\mapsto{}^\perp\FF:=\{X\in\mod \Lambda\mid\Hom_\Lambda(X,\FF)=0\}.\]
\item[(b)] \cite{S} They induce bijections between $\ftors \Lambda$ and $\ftorf \Lambda$.
\end{itemize}
\end{proposition}

Clearly $\tors \Lambda$ and $\ftors \Lambda$ have a structure of partially ordered sets
with respect to the inclusion relation.

\begin{definition}\label{definition of lattice}
Let $P$ be a partially ordered set and $x_i$ ($i\in I$) be elements in $P$.
If there exists a unique maximal element in the subposet $\{y\in P\mid y\le x_i,\ \forall i\in I\}$ of $P$,
we call it a \emph{meet} of $x_i$ ($i\in I$) and denote it by $\bigwedge_{i\in I}x_i$.
Dually we define a \emph{join} $\bigvee_{i\in I}x_i$.
We say that $P$ is a \emph{meet-semilattice} (respectively, \emph{join-semilattice}) if any finite subset of $P$ has a meet (respectively, join).
We say that $P$ is a \emph{lattice} if it is a join-semilattice and a meet-semilattice.
More strongly, we say that $P$ is a \emph{complete lattice} (respectively,
\emph{complete join-semilattice}, \emph{complete meet-semilattice})
if any subset of $P$ has a meet and a join (respectively, a join, a meet).

If a map $f:P\to P'$ between lattices preserves a join and a meet of any finite subset
(respectively, any subset), we call $f$ a \emph{morphism} of lattices
(respectively, complete lattices).
\end{definition}

We have the following statement.

\begin{proposition}\label{tors A is complete}
\begin{itemize}
\item[(a)] $\tors \Lambda$ and $\torf \Lambda$ are complete lattices, and we have an isomorphism $\tors \Lambda\to(\torf \Lambda)^{\op}$, $\TT\mapsto\TT^\perp$ of complete lattices.
\item[(b)] For torsion classes $\TT_i$ ($i\in I$) in $\mod \Lambda$, we have
\[\bigwedge_{i\in I}\TT_i=\bigcap_{i\in I}\TT_i\ \mbox{ and }\ \bigvee_{i\in I}\TT_i={}^\perp(\bigcap_{i\in I}\TT_i^\perp).\]
\item[(c)] For torsionfree classes $\FF_j$ ($j\in J$) in $\mod \Lambda$, we have
\[\bigwedge_{j\in J}\FF_j=\bigcap_{j\in J}\FF_j\ \mbox{ and }\ \bigvee_{j\in J}\FF_j=(\bigcap_{j\in J}{}^\perp\FF_j)^\perp.\]
\end{itemize}
\end{proposition}

\begin{proof}
It is clear that a meet of torsion classes $\TT_i$ ($i\in I$) is given by $\bigcap_{i\in I}\TT_i$.
Dually a meet of torsionfree classes $\FF_j$ ($j\in J$) is clearly given by $\bigcap_{j\in J}\FF_j$.

It is also clear that the bijection in Proposition \ref{tors and torf} gives an isomorphism $\tors \Lambda\to(\torf \Lambda)^{\op}$ of partially ordered sets.
Hence ${}^\perp(\bigcap_{i\in I}\TT_i^\perp)$ gives a join of $\TT_i$ ($i\in I$),
and $\bigvee_{j\in J}\FF_j=(\bigcap_{j\in J}{}^\perp\FF_j)^\perp$ gives a join of $\FF_j$ ($j\in J$).
\end{proof}

\begin{proposition}\label{opposite}
Let $\Lambda$ be a finite dimensional $k$-algebra. Then 
\begin{itemize}
\item[(a)] We have an isomorphism of complete lattices:
\[\tors\Lambda\to(\tors(\Lambda^{\op}))^{\op},\ \ \ \TT\mapsto D(\TT^\perp).\]
\item[(b)] The map in (a) induces a bijection $\ftors\Lambda\to\ftors(\Lambda^{\op})$.
In particular, $\ftors \Lambda$ forms a meet-semilattice (respectively, complete meet-semilattice)
if and only if $\ftors(\Lambda^{\op})$ forms a join-semilattice (respectively, complete join-semilattice).
\end{itemize}
\end{proposition}

\begin{proof}
(a) We have an isomorphism $\torf \Lambda\to\tors(\Lambda^{\rm op})$, $\FF\mapsto D(\FF)$ of complete lattices.
Thus the assertion follows from Proposition \ref{tors A is complete}.

(b) This follows from Proposition \ref{tors and torf} since $\FF$ is functorially finite if and only if so is $D(\FF)$.
\end{proof}

We now show that $\ftors \Lambda$ being a lattice is preserved by factoring by ideals $\left<e\right>$,
where $e$ is an idempotent element in $\Lambda$.

\begin{proposition}\label{interval}
Let $\Lambda$ be a finite dimensional $k$-algebra, and $e$ an idempotent in $\Lambda$.
\begin{itemize}
\item[(a)]  $\ftors(\Lambda/\langle e\rangle)$ is the interval
$\{ \TT\in\ftors\Lambda\mid0\subseteq\TT\subseteq\mod(\Lambda/\langle e\rangle)\}$ in $\ftors\Lambda$.
\item[(b)] If $\ftors \Lambda$ is a lattice (respectively, complete lattice), then $\ftors(\Lambda/\langle e \rangle)$ is a lattice (respectively, complete lattice).
\end{itemize}
\end{proposition}

\begin{proof}
(a) This is shown in \cite[Theorem 2.7]{AIR} and \cite[Proposition 2.27]{AIR}.

(b) This is a consequence of (a), using that an interval of a lattice (respectively, complete lattice) is again a lattice (respectively, complete lattice).
 \end{proof}

\subsection{Proof of Theorem \ref{complete lattice}}

We denote by $\tau$ the Auslander-Reiten translation of $\Lambda$.

\begin{definition}\label{define tau-rigid}
\begin{itemize}
\item[(a)] We call $M\in\mod \Lambda$ \emph{$\tau$-rigid} if $\Hom_\Lambda(M,\tau M)=0$.
We call $M\in\mod \Lambda$ \emph{$\tau$-tilting} if it is $\tau$-rigid
and $|M|=|\Lambda|$ holds, where $|M|$ is the number of non-isomorphic
indecomposable direct summands of $M$.
\item[(b)] We call $M\in\mod \Lambda$ \emph{support $\tau$-tilting} if
there exists an idempotent $e$ of $\Lambda$ such that $M$ is a $\tau$-tilting
$(\Lambda/\langle e\rangle)$-module.
\end{itemize}
\end{definition}

We denote by $\sttilt \Lambda$ the set of isomorphism classes of basic
support $\tau$-tilting $\Lambda$-modules.
Then we have the following result.

\begin{proposition}\label{basic bijection}\cite[Theorem 2.7]{AIR}
There exists a bijection $\sttilt \Lambda\to\ftors \Lambda$ given by $M\mapsto\Fac M$.
\end{proposition}

Using the bijection in Proposition \ref{basic bijection}, we regard
$\sttilt \Lambda$ as a partially ordered set which is isomorphic to $\ftors \Lambda$.

\begin{definition}\label{tau-rigid finite}\cite{DIJ}
We say that $\Lambda$ is \emph{$\tau$-rigid finite} if there are only finitely many indecomposable
$\tau$-rigid $\Lambda$-modules.
This is equivalent to $|\sttilt \Lambda|<\infty$, and to $|\ftors \Lambda|<\infty$.
\end{definition}

For example, any local algebra is $\tau$-rigid finite. In fact $\sttilt \Lambda=\{\Lambda,0\}$ holds in this case.
A path algebra $kQ$ of an acyclic quiver $Q$ is $\tau$-rigid finite if and only if $Q$ is a Dynkin quiver.
On the other hand, any preprojective algebra of Dynkin type is $\tau$-rigid finite \cite{M}.

We say that two non-isomorphic basic support $\tau$-tilting $\Lambda$-modules $M$ and $N$ are \emph{mutations} of each other if $M=X\oplus U$, $N=Y\oplus U$ and $X$ and $Y$ are either $0$ or indecomposable.
Then any support $\tau$-tilting $\Lambda$-module has exactly $n$ mutations.

The following results play a crucial role.

\begin{proposition}\label{infinite chain}
Let $\Lambda$ be a finite dimensional $k$-algebra.
\begin{itemize}
\item[(a)] \cite[Theorem 2.35]{AIR} If $M$ and $N$ are support $\tau$-tilting $\Lambda$-modules such that $M>N$,
then there exists a mutation $L$ of $N$ such that $M\ge L>N$.
\item[(b)] \cite[Proposition 3.2]{DIJ} Assume that $\Lambda$ is not $\tau$-rigid finite.
Then there exists an infinite descending chain of mutations
$\Lambda=M_0>M_1>M_2>\cdots$.
\item[(c)] \cite[Theorem 3.1]{DIJ} $\Lambda$ is $\tau$-rigid finite if and only if every torsion class in $\mod A$ is functorially finite.
\end{itemize}
\end{proposition}

Now we are ready to prove Theorem \ref{complete lattice}.

(d)$\Rightarrow$(a) This is immediate from Proposition \ref{tors A is complete}(a).

(a)$\Rightarrow$(c) This is clear.

(c)$\Rightarrow$(e) We assume that $\ftors \Lambda$ is a complete meet-semilattice and that $\Lambda$ is not $\tau$-rigid finite.
Take an infinite descending chain in Proposition \ref{infinite chain}(b).
Since $\sttilt \Lambda\simeq\ftors \Lambda$ is a complete meet-semilattice by our assumption, there exists a meet $M$ of $M_i$ ($i\ge0$) in $\sttilt \Lambda$.
Let $N_1,\ldots,N_n$ be all mutations of $M$.
Since $\Fac M_i\supsetneq\Fac M$, the set $I_i:=\{1\le k\le n\mid M_i\ge N_k>M\}$ is non-empty by Proposition \ref{infinite chain}(a).
Since we have a descending chain
\[I_0\supset I_1\supset I_2\supset\cdots\]
of finite non-empty sets, their intersection $I:=\bigcap_{i\ge0}I_i$ is also non-empty.
Then any $k\in I$ satisfies $M_i\ge N_k>M$ for all $i$.
This is a contradiction since $M$ is a meet of $M_i$ ($i\ge0$).

(e)$\Rightarrow$(d) This follows from Proposition \ref{infinite chain}(c).

(a)$\Leftrightarrow$(b) We have already shown that the conditions (a), (c), (d) and (e) are equivalent.
Replacing $\Lambda$ by $\Lambda^{\op}$, we have that 
(a) for $\Lambda^{\op}$ is equivalent to (c) for $\Lambda^{\op}$.
Using Proposition \ref{opposite}(b), we have the assertion.
\qed

\subsection{A criterion for the existence of joins and meets}
In this subsection, we need the following result, which improves Proposition \ref{infinite chain}(a).

\begin{proposition}\cite[Theorem 3.3]{DIJ}\label{result from DIJ}
Let $M$ be a support $\tau$-tilting $\Lambda$-module and $\TT$ a torsion class in $\mod \Lambda$.
\begin{itemize}
\item[(a)] If $\Fac M\supsetneq\TT$, then there exists a mutation $N$ of $M$ satisfying $\Fac M\supsetneq\Fac N\supset\TT$.
\item[(b)] If $\Fac M\subsetneq\TT$, then there exists a mutation $N$ of $M$ satisfying $\Fac M\subsetneq\Fac N\subset\TT$.
\end{itemize}
\end{proposition}

Immediately we have the following property of non-functorially finite torsion classes.

\begin{proposition}\label{non-functorially finite}
Let $\Lambda$ be a finite dimensional $k$-algebra, and $\TT$ a torsion class in $\mod \Lambda$
which is not functorially finite.
\begin{itemize}
\item[(a)] For any $\TT'\in\ftors \Lambda$ satisfying $\TT'\supsetneq\TT$, there exists $\TT''\in\ftors \Lambda$
satisfying $\TT'\supsetneq\TT''\supset\TT$.
\item[(b)] For any $\TT'\in\ftors \Lambda$ satisfying $\TT'\subsetneq\TT$, there exists $\TT''\in\ftors \Lambda$
satisfying $\TT'\subsetneq\TT''\subset\TT$.
\end{itemize}
\end{proposition}

\begin{proof}
The statement follows immediately from Propositions \ref{basic bijection} and \ref{result from DIJ}.
\end{proof}

We give a more explicit criterion for existence of a meet and a join.

\begin{theorem}\label{existence of meet}
Let $\Lambda$ be a finite dimensional $k$-algebra.
\begin{itemize}
\item[(a)] A subset $\{\TT_i\mid i\in I\}$ of $\ftors \Lambda$ has a meet
if and only if $\bigcap_{i\in I}\TT_i$ is functorially finite.
\item[(b)] A subset $\{\TT_i\mid i\in I\}$ of $\ftors \Lambda$ has a join
if and only if ${}^\perp(\bigcap_{i\in I}\TT_i^\perp)$ is functorially finite.
\end{itemize}
\end{theorem}

\begin{proof}
We only have to prove (a) since (b) is a dual.

If $\bigcap_{i\in I}\TT_i$ is functorially finite, then it is a meet
of $\TT_i$ ($i\in I$) in $\ftors \Lambda$, by Proposition \ref{tors A is complete}.
Thus we only have to prove the `only if' part.

Assume that $\TT_i$ ($i\in I$) has a meet $\SS$ in $\ftors \Lambda$ and
that $\TT:=\bigcap_{i\in I}\TT_i$ is not functorially finite.
Since $\SS\subset\TT_i$ for all $i\in I$, we have $\SS\subset\TT$.
Since $\TT$ is not functorially finite, we have $\SS\subsetneq\TT$. 
Applying Proposition \ref{non-functorially finite}, there exists $\SS'\in\ftors \Lambda$ such that
\[\SS\subsetneq\SS'\subset\TT.\]
Thus $\SS'\subset\TT_i$ holds for any $i\in I$.
This is a contradiction since $\SS$ is a meet of $\TT_i$ ($i\in I$).
\end{proof}

\begin{remark} \label{tors/ftors}
The statements in the above theorem mean that a meet (respectively, join) in $\ftors \Lambda$ has to be the same as a meet (respectively, join) in the complete lattice $\tors \Lambda$.
\end{remark}

\section{Lattice structure of torsion classes for path algebras}

\subsection{Sufficient conditions for $\ftors(kQ)$ to be a lattice}

Let $Q$ be a finite connected acyclic quiver. In this section we give two sufficient conditions for $\ftors(kQ)$ to be a lattice.
Since for an artin algebra of finite representation type any subcategory is functorially finite, the first result is a direct consequence of the fact that $\tors(kQ)$ is a lattice.

\begin{proposition} \label{Dynkin} If $Q$ is a Dynkin diagram, then $\ftors(kQ)$ is a lattice.
\end{proposition}

When $Q$ is a Dynkin diagram, the lattice $\ftors(kQ)$ was shown in
\cite[Theorem 4.3]{IT} to be a Cambrian lattice in the sense of Reading \cite{Re} .

The second sufficient condition is the following.

\begin{proposition} \label{2vertices} Assume that $Q$ has at most two vertices. Then $\ftors(kQ)$ is a lattice.
\end{proposition}
\begin{proof} If $Q$ has one vertex, then $kQ\cong k$, hence the claim is obvious.
Assume then that we have two vertices. Then our quiver $Q$ is 
$1 \xrightarrow{(n)} 2$, with $n\geq 2$ arrows. We can assume $n\geq 2$ since otherwise $Q$ is Dynkin. The Auslander-Reiten quiver is then of the form:
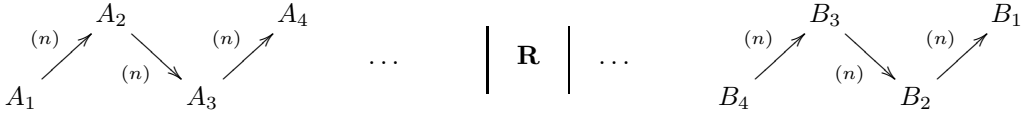
\[\xymatrixrowsep{2pt}\xymatrixcolsep{5pt}
\xymatrix{
&&A_2\ar[ddrr]_{(n)}&&&&A_4&&&&&&&&&&&&&B_3\ar[rrdd]_{(n)}&&&&B_1&&&&\\
&&&&&&&&\dots&&&& \RR  &&\dots&&&&&&&&&&&&&&&&&&&&&&&&&&&&&&\\
A_1\ar[uurr]^{(n)}&&&&A_3\ar[uurr]^{(n)}&&&&&&&\ar@{-}[uu] & &\ar@{-}[uu] &&&&B_4\ar[uurr]^{(n)}&&&&B_2\ar[uurr]^{(n)}&&&
}\]
Here $\RR$ consists of tubes when $n=2$, and of $\mathbb Z A_{\infty}$-components when $n>2$. It is known that no indecomposable rigid module lies in $\RR$.
The tilting modules are given by two consecutive vertices in the preprojective or preinjective component.
So for $i\geq 2$ we have the tilting module $A_{i-1}\oplus A_i$, with associated torsion class $\TT_i=\Fac(A_{i-1}\oplus A_i)$ which is equal to $\Fac A_{i-1}$ when $i\geq 3$.
For $i\geq 2$ we have the tilting modules $B_i\oplus B_{i-1}$ with associated torsion class $\TT_i'=\Fac(B_i\oplus B_{i-1})=\Fac B_i$.
There are no other tilting modules. The additional support tilting modules are the simple modules $A_1$ and $B_1$, and hence we have the additional torsion classes $\TT_1=\Fac A_1$ and $\TT_1'=\Fac B_1$.

We have the inclusions $\{0\}\subset\TT_1\subset \TT_2 \supset \TT_3 \supset \dots \supset  \TT_i \supset \dots \supset \TT_j' \supset \dots\supset \TT_1'$
for all elements of $\ftors (kQ)$. It is clear that if neither $\TT$ nor $\TT'$ is $\TT_1$, then $\TT\vee \TT'$  is the larger one and $\TT\wedge \TT'$ is the smaller one.
Further, $\TT_1\vee \TT=\TT_2 (= \mod kQ)$  for  $\TT \neq \TT_1$, and  $\TT_1\wedge \TT_2=\TT_1$, $\TT_1\wedge \TT=\{0\}$ for $\TT\neq \TT_2$.
 \end{proof}

\subsection{Tame algebras}
In this section we deal with path algebras $kQ$ of extended Dynkin quivers with at least 3 vertices, and show that in that case the $\ftors (kQ)$ do not form lattices.

\begin{proposition}\label{tame}
Let $Q$ be an acyclic extended Dynkin quiver with at least 3 vertices. Then $\ftors(kQ)$ is neither a join-semilattice nor a meet-semilattice.
\end{proposition}

\begin{proof} Since $kQ$ is extended Dynkin with at least 3 vertices, there is a tube $\CC$ of rank $r\geq2$ and there are $r$ quasi-simple modules $S_1,  \dots , S_r$ in $\CC$. Since  $S_1,  \dots , S_r$   are $\tau$-rigid, we have that  
$\TT_1=\Fac S_1, \dots , \TT_r=\Fac S_r$ are in $\ftors (kQ)$. By Theorem \ref{existence of meet} there is a  join of these $\TT_i$ in $\ftors (kQ)$ if and only if ${}^\perp(\bigcap_{i\in I}\TT_i^\perp)$ is functorially finite, where $I=\{1,\dots, n\}$. However ${}^\perp(\bigcap_{i\in I}\TT_i^\perp)= \add(\CC\cup \{$preinjectives$\})$ which is not functorially finite, since it clearly cannot be written as $\Fac Y$ for any $Y$. Therefore there is no join in $\ftors(kQ)$, and hence $\ftors(kQ)$ is not a join-semilattice.

Since $Q^{\rm op}$ is an acyclic extended Dynkin quiver with at least 3 vertices, $\ftors(kQ^{\rm op})$ is not a join-semilattice.
By Proposition \ref{opposite}, $\ftors(kQ)$ is not a meet-semilattice.
\end{proof}

\subsection{Wild algebras}

In this section we show that $\ftors(kQ)$ is not a lattice when the quiver $Q$ is connected wild, with 3 vertices.

For a finite dimensional algebra $\Lambda$ and a set $S$ of $\Lambda$-modules,
we denote by $\Filt S$ the full subcategory of $\mod\Lambda$ whose objects are
the $\Lambda$-modules which have a finite filtration with factors in $S$.

\begin{proposition}\label{R}
Let $Q$ be an acyclic quiver, and let $M$ and $N$ be indecomposable rigid $kQ$-modules such that $\Hom_{kQ}(M,N)=0=\Hom_{kQ}(N,M)$, $\Ext_{kQ}^1(M,N)\neq0$ and $\Ext_{kQ}^1(N,M)\neq0$.
\begin{itemize}
\item[(a)] \cite{R} The category $\AA:=\Filt(M,N)$ is an exact abelian subcategory of $\mod kQ$ with two simple objects $M$ and $N$.
\item[(b)] $\End_{kQ}(M)\cong k\cong \End_{kQ}(N)$ holds, and $M$ and $N$ are regular.
\item[(c)] For any $\ell\ge0$, there exists an object $X_\ell$ in $\AA$ which is uniserial of length $\ell$ in $\AA$.
\end{itemize}
\end{proposition}

\begin{proof}
(a) This is shown in \cite[Theorem 1.2]{R}.

(b) Since $M$ and $N$ are rigid, we have the first assertion.
Since $\Ext_{kQ}^1(M,N)\neq0$ and $\Ext_{kQ}^1(N,M)\neq0$ hold,
$M$ and $N$ are in a cycle. Hence they are regular.

(c) The assertion is clear for $\ell=1$.
Assume that we have a uniserial object $X_\ell$ of length $\ell$ in $\AA$.
Without loss of generality, let $M$ be the top of $X_\ell$ in $\AA$.
Then there exists an exact sequence $0\to \rad_{\AA}X_\ell\to X_\ell\to M\to0$.
Since $\Ext^1_{kQ}(N,M)\neq0$, there exists a non-split exact sequence
$0\to M\to E\to N\to 0$. Since $kQ$ is hereditary, we have a commutative diagram
of exact sequences:
\begin{equation*}
\xymatrixrowsep{7pt}\xymatrixcolsep{20pt}
\xymatrix{
&0&0\\
0\ar[r]&M\ar[r]\ar[u]&E\ar[r]\ar[u]&N\ar[r]&0\\
0\ar[r]&X_\ell\ar[r]\ar[u]&Y\ar[r]\ar[u]&N\ar[r]\ar@{=}[u]&0\\
&\rad_{\AA}X_\ell\ar@{=}[r]\ar[u]&\rad_{\AA}X_\ell\ar[u]\\
&0\ar[u]&0.\ar[u]
}\end{equation*}
Clearly $Y$ belongs to the category $\AA$.
We show that $Y$ is uniserial of length $\ell+1$ in $\AA$.
It is enough to show $\rad_{\AA}Y=X_\ell$. Otherwise $\rad_{\AA}Y$ is strictly contained
in $X_\ell$, and hence $\rad_{\AA}Y=\rad_{\AA}X_\ell$ holds since $X_\ell$ is uniserial.
Then $Y/\rad_{\AA}Y=E$ holds, a contradiction since $E$
is not semisimple in the category $\AA$. Thus the assertion follows.
\end{proof}

We shall also need the following.

\begin{lemma} \label{quotient}Let $\mathcal C$ be a full subcategory of $\mod kQ$ closed under extensions. Then $\Fac \mathcal C$ is also closed under extensions.
\end{lemma}
\begin{proof} Let $0 \rightarrow X \rightarrow Y \rightarrow Z \rightarrow 0$ be an exact sequence in $\mod kQ$, where $X$ and $Z$ are in $\Fac \mathcal C$. Then we have surjections $f: C_0\rightarrow X$ and  $g: C_1\rightarrow Z$, where $C_0$ and $C_1$ are in $\mathcal C$. This gives rise to the exact sequence:
\begin{eqnarray*}
\Ext_{kQ}^1(C_1,C_0)\rightarrow  \Ext_{kQ}^1(C_1,X)\rightarrow \Ext_{kQ}^2(C_1,\Ker f)=0.
\end{eqnarray*}
Thus we get the exact commutative diagrams:
$$\xymatrixrowsep{7pt}\xymatrixcolsep{20pt}
\xymatrix{0\ar[r]&X\ar[r]\ar@{=}[d]&Y'\ar[r]\ar[d]&C_1\ar[r]\ar[d]&0\\
0\ar[r]&X\ar[r]&Y\ar[r]&Z\ar[r]&0
}\qquad\qquad
\xymatrixrowsep{7pt}\xymatrixcolsep{20pt}
\xymatrix{0\ar[r]&C_0\ar[r]\ar[d]&Y''\ar[r]\ar[d]&C_1\ar[r]\ar@{=}[d]&0\\
0\ar[r]&X\ar[r]&Y'\ar[r]&C_1\ar[r]&0
}$$

\noindent Since $\mathcal C$ is extension closed, then $Y''$ is in $\mathcal C$, and we have surjections $Y''\rightarrow Y'\rightarrow Y$, so that $Y$ is in $\Fac \mathcal C$, as desired.
\end{proof}

Combining the above results, we get the following.

\begin{proposition}\label{wild}
Let $kQ$ be an acyclic quiver, and $M$ and $N$ be $kQ$-modules satisfying the assumptions in Proposition \ref{R}.
Let $\AA:=\Filt(M,N)$ and $\TT:= \Fac\AA$. Then:
\begin{itemize}
\item[(a)] The subcategory $\TT$ is a torsion class which is not functorially finite in $\mod (kQ)$.
\item[(b)] $\ftors (kQ)$ is neither a join-semilattice nor a meet-semilattice.
\end{itemize}
\end{proposition}

\begin{proof}
(a) It follows from Lemma \ref{quotient} that $\TT$ is a torsion class.
Let $\TT_1:=\Fac M$ and $\TT_2:=\Fac N$. Since $M$ and $N$ are rigid, the subcategories  $\TT_1$ and  $\TT_2$ are in $\ftors (kQ)$. 

Assume that $\TT$ is functorially finite. Then there exists a module $X$ in $\TT$ so that $\TT=\Fac X$.
By the definition of $\TT$, there is a module $C$ in $\AA$ and an epimorphism $C\rightarrow X$ in $\mod kQ$, and hence $\TT=\Fac C$.
Now let $\ell$ be the Loewy length of $C$ in $\AA$.
Since the modules $M$ and $N$ satisfy the conditions of Proposition \ref{R},
there is a uniserial object $X_{\ell+1}$ of length $\ell+1$ in $\AA$.
Since $X_{\ell+1}\in\Fac C$, there is an epimorphism $C^m\to X_{\ell+1}$ in
$\mod kQ$ (and hence in $\AA$) for some $m\ge0$.
This is a contradiction since the Loewy length of $X_{\ell+1}$ is bigger than that
of $C$.

(b) If $\ftors(kQ)$ is a lattice, we know from section 1 that the join of $\Fac M$ and $\Fac N$ must be the smallest torsion class containing $\Fac M$ and $\Fac N$, which is clearly $\TT$. But since we have seen that this is not a functorially finite subcategory of $\mod kQ$ by (a), it follows that $\ftors(kQ)$ is not a join-semilattice. 

Since the $kQ^{\rm op}$-modules $DM$ and $DN$ satisfy the conditions of Proposition \ref{R}, we have that $\ftors(kQ^{\rm op})$ is not a join-semilattice.
By Proposition \ref{opposite}, $\ftors(kQ)$ is not a meet-semilattice.
\end{proof}

Now we are able to show the following result, where $\xymatrix{1\ar[r]^{(a)}&2}$ denotes $a$ multiple arrows from $1$ to $2$.

\begin{lemma}\label{1/6}
Let $Q=\xymatrix{1\ar@<.5mm>[r]^{(a)}\ar@/_1pc/[rr]_{(c)}&2\ar@<.5mm>[r]^{(b)}&3}$ be a quiver with $a\ge 2$, $b\ge1$ and $c\ge0$. Then there exist $M$ and $N$ satisfying the conditions in Proposition \ref{R}.
\end{lemma}

\begin{proof}
Let $Q':=(1 \xrightarrow {(a)}2)$ be a full subquiver of $Q$.
We regard the projective $kQ'$-module corresponding to the
vertex $1$ as a $kQ$-module $M$, and let $N:=\tau_{kQ}M$.
We show that $M$ and $N$ satisfy the conditions in Proposition \ref{R} with $p>0$ and $q>0$.
We have $\dimvec M=(1,a,0)^t$. Since the Cartan matrix of $kQ$ (see \cite{ASS}) is
$C=\left[\begin{smallmatrix}
1&0&0\\ a&1&0\\ ab+c&b&1
\end{smallmatrix}\right]$ 
and the Coxeter matrix of $kQ$ (see \cite{ASS}) is given by
\[\Phi=-C^t\cdot C^{-1}=-\left[\begin{smallmatrix}
1&a&ab+c\\ 0&1&b\\ 0&0&1
\end{smallmatrix}\right]
\left[\begin{smallmatrix}                                                         
1&0&0\\ -a&1&0\\ -c&-b&1
\end{smallmatrix}\right]
=\left[\begin{smallmatrix}                                                         
a^2+abc+c^2-1&ab^2+bc-a&-ab-c\\ a+bc&b^2-1&-b\\ c&b&-1
\end{smallmatrix}\right],\]
we have
\[\dimvec N=\Phi\cdot\dimvec M=(a^2b^2+2abc+c^2-1,ab^2+bc,ab+c)^t.\]

\noindent{\rm (Step 1) }
Since $M$ is a rigid $kQ'$-module and $Q'$ is a full subquiver of $Q$,
it is a rigid $kQ$-module.
Hence $N$ is also a rigid $kQ$-module since $\tau$ preserves the
rigidity of $kQ$-modules.

Since $M$ is rigid, we have $\Hom_{kQ}(M,N)=\Hom_{kQ}(M,\tau M)=0$.
We have $\Ext^1_{kQ}(M,N)=\Ext^1_{kQ}(M,\tau M)\simeq D\underline{\End}_{kQ}(M)\neq0$ by Auslander-Reiten duality.

It remains to show that $\Hom_{kQ}(N,M)=0$ and $\Ext^1_{kQ}(N,M)\neq0$.

\noindent{\rm (Step 2) }
To prove $\Hom_{kQ}(N,M)=0$, it is enough to show $\Hom_{kQ}(M,\tau^{-1}M)=0$.
Since $M$ does not have $S_3$ as a composition factor, it is enough to show that $\soc\tau^{-1}M$ is a direct sum of copies of $S_3$.
Since $S_1$ is injective, it does not appear in $\soc\tau^{-1}M$ by the indecomposability of $\tau^{-1}M$. 

Assume that $S_2$ appears in $\soc\tau^{-1}M$.
Then we have an exact sequence $0\to S_2\to \tau^{-1}M\to L\to0$.
Applying $\Hom_{kQ}(-,S_3)$, we have an exact sequence
\[\Ext^1_{kQ}(\tau^{-1}M,S_3)\to\Ext^1_{kQ}(S_2,S_3)\to\Ext^2_{kQ}(L,S_3)=0.\]
Since $\Ext^1_{kQ}(S_2,S_3)\neq0$, we have $\Ext^1_{kQ}(\tau^{-1}M,S_3)\neq0$.
On the other hand, we have by Auslander-Reiten duality,
\[\Ext^1_{kQ}(\tau^{-1}M,S_3)\simeq D\Hom_{kQ}(S_3,M)=0,\]
a contradiction.

\noindent{\rm (Step 3) }
To prove $\Ext^1_{kQ}(N,M)\neq0$, we calculate the Euler form, see \cite{ASS}.
We have
\begin{eqnarray*}
\langle N,M\rangle=(\dimvec N)^t\cdot(C^{-1})^t\cdot\dimvec M&=&
(a^2b^2+2abc+c^2-1,ab^2+bc,ab+c)\left[\begin{smallmatrix}                               
1&-a&-c\\ 0&1&-b\\ 0&0&1
\end{smallmatrix}\right]
\left[\begin{smallmatrix}1\\ a\\ 0\end{smallmatrix}\right]\\
&=&-1-a^2(a^2b^2-2b^2-1)-abc(2a^2-3)-c^2(a^2-1),
\end{eqnarray*}
which is easily shown to be negative by our assumption $a\ge2$, $b\ge1$ and $c\ge0$.
\end{proof}

Now we show the following main result in this section.

\begin{proposition} \label{wild3}
Let $Q$ be a connected acyclic wild quiver with 3 vertices. Then:
\begin{itemize}
\item[(a)] There exist $M$ and $N$ satisfying the conditions in Proposition \ref{R}.
\item[(b)] $\ftors(kQ)$ is neither a join-semilattice nor a meet-semilattice.
\end{itemize}
\end{proposition}

\begin{proof}
(a) Let $a,b,c$ be integers such that $a\ge2$, $b\ge1$ and $c\ge0$.
Then $Q$ has one of the following forms:
\begin{eqnarray*}
{\rm (i)}: \xymatrix{1\ar@<.5mm>[r]^{(a)}\ar@/_1pc/[rr]_{(c)}&2\ar@<.5mm>[r]^{(b)}&3},\ \ \ 
{\rm (ii)}: \xymatrix{1\ar@<.5mm>[r]^{(b)}\ar@/_1pc/[rr]_{(a)}&2\ar@<.5mm>[r]^{(c)}&3},\ \ \ 
{\rm (iii)}: \xymatrix{1\ar@<.5mm>[r]^{(c)}\ar@/_1pc/[rr]_{(b)}&2\ar@<.5mm>[r]^{(a)}&3},\ \ \ \\
{\rm (iv)}: \xymatrix{1\ar@<.5mm>[r]^{(b)}\ar@/_1pc/[rr]_{(c)}&2\ar@<.5mm>[r]^{(a)}&3},\ \ \ 
{\rm (v)}: \xymatrix{1\ar@<.5mm>[r]^{(c)}\ar@/_1pc/[rr]_{(a)}&2\ar@<.5mm>[r]^{(b)}&3},\ \ \ 
{\rm (vi)}: \xymatrix{1\ar@<.5mm>[r]^{(a)}\ar@/_1pc/[rr]_{(b)}&2\ar@<.5mm>[r]^{(c)}&3}.
\end{eqnarray*}
First, the case (i) was shown in Lemma \ref{1/6}.
Next, the case (ii) (respectively, (iii)) follows from the case (i) by using the reflection functor at the vertex 1 (respectively, 3).
Finally the case (iv) (respectively, (v), (vi)) follows from the case (i) (respectively, (ii), (iii)) by using the $k$-dual.

(b) This follows from (a) and Proposition \ref{wild}.
\end{proof}

\begin{remark}
When $a,b,c\ge1$, it is easy to check that the modules $M=S_2$ and $N:=\xymatrix{k\ar@<.5mm>[r]^{(a)}\ar@/_2pc/[rr]^{(c)}_{f=(1,\ldots,1)}&0\ar@<.5mm>[r]^{(b)}&k^c}$ also satisfy the conditions in Proposition \ref{R} with $p>0$ and $q>0$.
\end{remark}

\subsection{Proof of Theorem \ref{main}}

We need the following preparation, which is an analog of a well-known result, see \cite[Lemma VII.2.1]{ASS}.

\begin{proposition}\label{four cases}
Let $Q$ be a finite connected quiver. Then one of the following holds.
\begin{itemize}
\item[(a)] $Q$ is a Dynkin quiver.
\item[(b)] $Q$ has at most two vertices.
\item[(c)] $Q$ has an extended Dynkin full subquiver with at least 3 vertices.
\item[(d)] $Q$ has a connected wild full subquiver with exactly 3 vertices.
\end{itemize}
\end{proposition}

\begin{proof}
First, assume that $Q$ has multiple arrows from $i$ to $j$.
If $Q$ has exactly two vertices, then we have the case (b).
If $Q$ has at least 3 vertices, then any connected full subquiver of $Q$ consisting of $i$, $j$ and one more vertex is wild. Thus we have the case (d).

Next, assume that $Q$ has no multiple arrows. Then it follows from \cite[Lemma VII.2.1]{ASS} that we have either the case (a) or (c).
\end{proof}

Now we are ready to prove Theorem \ref{main}.

(d)$\Rightarrow$(a) If $Q$ is a Dynkin quiver, then $\ftors(kQ)$ forms a lattice by Proposition \ref{Dynkin}.
If $Q$ has exactly two vertices, then $\ftors(kQ)$ forms a lattice by Proposition \ref{2vertices}.

(a)$\Rightarrow$(b) This is clear.

(b)$\Rightarrow$(d) Assume that $Q$ does not satisfy the condition (d).
Then by Proposition \ref{four cases}, $Q$ has either an extended Dynkin full subquiver with at least 3 vertices, or a connected wild full subquiver with exactly 3 vertices,
For the former case (respectively, latter case), $\ftors (kQ)$ is not a join-semilattice by Propositions \ref{tame} (respectively, \ref{wild3}) and \ref{interval}(b).

(c)$\Leftrightarrow$(d) By Proposition \ref{opposite}, the condition (c) is equivalent to that $\ftors(kQ^{\rm op})$ forms a join-semilattice.
This is equivalent to that $Q^{\rm op}$ is either a Dynkin quiver or has at most two vertices, by using the equivalence (b)$\Rightarrow$(d) for the quiver $Q^{\rm op}$.
This is clearly equivalent to the condition (d).
\qed

\subsection{Concealed canonical algebras and tubular algebras}

Inspired by the proof that $\ftors\Lambda$ is not a join-semilattice for path
algebras of extended Dynkin quivers with at least 3 vertices, we have the following.

\begin{proposition}\label{key for canonical}
Let $\Lambda$ be a finite dimensional $k$-algebra such that the set of indecomposable
$\Lambda$-modules is a disjoint union $\PP\cup\RR\cup\QQ$, 
where $\RR$ is a family of stable standard orthogonal tubes,
$\Hom_\Lambda(\RR,\PP)=0$, $\Hom_\Lambda(\QQ,\RR)=0$ and $\Hom_\Lambda(\QQ,\PP)=0$.
If there is a tube $\CC$ in $\RR$ of rank $r\ge2$, then $\ftors\Lambda$
is neither a join-semilattice nor a meet-semilattice.
\end{proposition}

\begin{proof}
We only prove the assertion for join-semilattices since the other assertion follows
by Proposition \ref{opposite}.

Let $S_1,\ldots,S_r$ be the indecomposable modules at the border of $\CC$.
Since $\CC$ is standard, then $S_1,\ldots,S_r$ are
$\tau$-rigid, and hence $\Fac S_i$ is in $\ftors\Lambda$ for $i=1,\ldots,n$.
Let $\TT:=\bigvee_{i=1}^n\Fac S_i$ in $\tors\Lambda$. Then $\TT$ is the smallest
torsion class in $\mod\Lambda$ containing $\CC$.
Since $\add(\CC,\QQ)$ is a torsion class by our assumptions, we have $\TT
\subset\add(\CC,\QQ)$.
Now if $\TT$ is functorially finite, then there exists $M\in\TT$ such that
$\TT=\Fac M$.
Since $\Hom_\Lambda(\QQ,\CC)=0$ holds by our assumption, the maximal direct summand
$N$ of $M$ contained in $\add\CC$ satisfies $\CC\subset\Fac N$.
But this is impossible since $\add\CC$ is equivalent to the category of finite
dimensional modules over the complete path algebra $k\widehat{Q}$ of the quiver
$Q$ of type $\widetilde{A}_{r-1}$ by our assumption,
and hence there is no upper bound of Loewy length of objects.
\end{proof}

Now we are ready to prove Theorem \ref{canonical algebra}.
It follows from Proposition \ref{key for canonical} and by the properties of the
concealed canonical (respectively, tubular algebras) listed in
\cite[page 380]{SS} (respectively, \cite[Theorem XIX.3.20]{SS})
since there exists a tube $\CC$ of rank $r\ge2$ if and only if
$\Lambda$ has at least 3 vertices.
\qed

\end{document}